\documentclass[12pt]{amsart}

\usepackage{amssymb,amsthm}

\usepackage{enumerate}

\makeatletter
\@namedef{subjclassname@2010}{%
  \textup{2010} Mathematics Subject Classification}
\makeatother

\newtheorem{thm}{Theorem}[section]

\newtheorem{lem}[thm]{Lemma}
\newtheorem{prop}[thm]{Proposition}
\newtheorem{ques}[thm]{Question}
\newtheorem{defi}[thm]{Definition}

\newtheorem*{Claim}{Claim}

\newtheorem*{mainthm}{Main Theorem}
\newtheorem*{thmm}{Theorem}

\theoremstyle{definition}

\newtheorem{rem}[thm]{Remark}

\newenvironment{demof}[1]
{\medbreak\noindent{\sc Proof of {#1} :}} {\hfill$\square$\medbreak}

\numberwithin{equation}{section}

\newcommand{\G}{\mathcal{G}}

\begin{document}
\baselineskip=17pt

\title[Jumps of entropy]{Jumps of entropy for $C^r$ interval maps}

\author[D. Burguet]{David Burguet}
\address{ LPMA - CNRS UMR 7599 \\ 
Universite Paris 6 \\
 75252 Paris Cedex 05 \\
  FRANCE}
\email{david.burguet@upmc.fr}

\date{}

\begin{abstract}
We study the jump of topological entropy for $C^r$ interval or circle maps. We prove in particular that the topological entropy is continuous at any $f\in C^r([0,1])$ with $h_{top}(f)>\frac{\log^+\|f'\|_\infty}{r}$. To this end we study the continuity of the entropy of the  Buzzi-Hofbauer diagrams associated to $C^r$ interval maps. 
\end{abstract}

\subjclass[2010]{Primary XXXX; Secondary YYYY}


\maketitle

\section{Introduction}

In this paper we study the upper semicontinuity of the (topological) entropy in the spaces $C^r([0,1])$ and $C^r(\mathbb{S}^1)$ of $C^r$ interval maps and circle maps   endowed with the usual $C^r$ topology where $r$ is a real number with $1\leq r \leq +\infty$. This problem of upper semicontinuity of the entropy has been investigated in several previous works in different settings \cite{Mis0}\cite{Mis1}\cite{Yoma}. Let us recall the main related results.

Lower semicontinuity of the entropy was proved by M.Misiurewicz and W.Szlenk \cite{Mis2}  for interval maps in the $C^0$ topology and by A.Katok for $C^{1+\alpha}$ surface diffeomorphisms in the $C^1$ topology \cite{Kat}. In both cases this follows from the characterization of  entropy by horseshoes which are persistent in the mentioned 
topologies. In dimension larger than two the entropy may not be  lower semicontinuous even in the $C^\infty$ topology \cite{Mis0}.

 In \cite{Mis1} M.Misiurewicz investigated upper semicontinuity of the entropy for continuous  piecewise monotone maps of the interval.  Let $\mathcal{M}^r_k([0,1])$ , with $r=0$ or $1$, be the set of $C^r$ interval maps $f$, which admit a partition of $[0,1]$ in $k$ intervals such that  $f$ is weakly monotone on each element of this partition. 

We say $x\in [0,1]$ is a turning point of an interval map $f$ when there exist $0\leq a<b\leq x\leq c<d\leq 1$ such that $f$ is constant on $[b,c]$ and strictly monotone 
both on $[a,b]$ and $[c,d]$ but in the opposite sense.
 
M.Misiurewicz proved upper semicontinuity (and thus continuity)  of the entropy in $\mathcal{M}^0_2([0,1])$ in the $C^0$ topology  at all maps at which it is positive. For all $k$ he  also gave a complete description of the possible jumps of entropy  in $\mathcal{M}^0_k([0,1])$ at  any $f\in \mathcal{M}^0_k([0,1])\setminus \bigcup_{l=1,...,k-1}\mathcal{M}^0_l([0,1])$:

\begin{eqnarray*}\limsup_{g\stackrel{C^0}{\rightarrow} f, \ g\in \mathcal{M}^0_k([0,1]}h_{top}(g)= \max\left(h_{top}(f),\beta(f)\right),
\end{eqnarray*}

with
\begin{eqnarray*}\beta(f)&:=\max\{ \frac{p}{q}\log 2, & \text{there exists a periodic point of $f$ of} \\
& & \text{ period $q$ with  $p(\leq q)$ turning points in its orbit}\}.
\end{eqnarray*} 
 
 Moreover M.Misiurewicz  proved  continuity of the entropy for the $C^1$ topology  for $C^1$ piecewise monotone maps with a uniform number of pieces \cite{Mis4}, i.e. 
 for any positive integer $k$ and for any $f\in  \mathcal{M}^1_k([0,1])$ we have:
 
\begin{eqnarray}\label{mi95}\limsup_{g\stackrel{C^1}{\rightarrow} f, \ g\in \mathcal{M}^1_k([0,1]}h_{top}(g)= h_{top}(f),
\end{eqnarray}

  Axiom A interval maps are open and dense in $C^r([0,1])$  with $r\geq 1$ \cite{Koz}. 
 We recall that a $C^1$ interval map $f$ is said to be Axiom A when all periodic points are hyperbolic and  when letting $B( f )$ denote the union of basins of attracting periodic points
of $f$, then $[0, 1]\setminus B( f )$ is a hyperbolic set, that is, there
are constants $C > 0$ and $\lambda > 1$, such that $|(f^k)'(x)|\geq C\lambda^k$
holds for all $x \in [0, 1]\setminus B( f )$ and $k\in \mathbb{N}$. By structural stability the entropy is locally constant, hence continuous,  on the set of Axiom A interval maps. 

For $r>1$ the set $D^r([0,1])$ of $C^r$ interval maps with no critical point flat up to order $r$ is also an open and dense set in $C^r([0,1])$ and neither contains the set of Axiom A maps nor is contained in it. The entropy is also continuous on $D^r([0,1])$.  This result is due to Bowen \cite{bo} and  Misiurewicz-Szlenk \cite{Mis2} for $r=2$ and to K.Iwai \cite{Iw}
 for larger $r$ (for the latter the proof is  based on a variant of the kneading theory of Milnor and Thurston). We see in the appendix that it is in fact a direct consequence of Misiurewicz's result (\ref{mi95}).

Upper semicontinuity was also established by Y.Yomdin for the $C^\infty$ topology on any compact manifold $M$ \cite{Yoma}. In fact he bounds the default of upper semicontinuity of the entropy for $C^r$ maps, $1\leq r \leq +\infty$,  as follows :
\begin{eqnarray}\label{yomd}
\limsup_{g\stackrel{C^r}{\rightarrow} f}h_{top}(g)\leq h_{top}(f)+\frac{d \cdot R(f)}{r},
\end{eqnarray}
where $d$ is the dimension of $M$ and $R(f):=\lim_n\frac{\log^+ \|Df^n\|_{\infty}}{n}$ (remark that $R$ is upper semicontinuous in the $C^1$ topology as the limit of a subadditive sequence of continuous functions, i.e. $\limsup_{g\stackrel{C^1}{\rightarrow} f}R(g)\leq R(f)$). 

Earlier M.Misiurewicz has proved that upper semicontinuity of the entropy fails for diffeomorphisms in the $C^r$ topology with finite $r$  in dimension larger than or equal to three \cite{Mis3}. For interval maps \cite{Buz1} and for surface diffeomorphisms \cite{Buz0} the only known  $C^r$ examples at which the entropy is not upper semicontinuous all satisfy $h_{top}(f)<\frac{R(f)}{r}$. 
We prove in this paper that it is always the case in dimension one.

\begin{mainthm}\nonumber
Let  $f$ be a $C^r$ interval or circle map with   $1\leq r\leq +\infty$. Then 

$$\limsup_{g \stackrel{C^r}{\rightarrow} f}h_{top}(g)\leq \max\left(h_{top}(f),\frac{R(f)}{r}\right).$$
\end{mainthm}

In fact we will prove the stronger statement where the $\limsup$ is taken over $g$ going to $f$ in the $C^1$ topology and staying in a $C^r$ bounded set (Yomdin's inequality (\ref{yomd}) also holds true in this setting). 

Obviously the statement of the Main Theorem for arbitrarily large $r$ implies the same for $r=+\infty$ and we recover in this last case the upper semicontinuity of the  entropy in the $C^\infty$ topology proved by Y.Yomdin. Note also that the above theorem is trivial for $r=1$ as $h_{top}(f)$ is always less than or equal to $d \cdot R(f)$ for any $C^1$ dynamical system on a compact manifold of dimension $d$.

We conjecture that the Main Theorem should also hold true for surface diffeomorphisms.  \\

Let $f$ be a $C^r$ interval map, $1< r<+\infty$, and let  $p$ be a repelling periodic point with period $T$.   Here $f$ may be noninvertible and the unstable manifold $W^u(\mathcal{O}(p))$ of the orbit $\mathcal{O}(p)$ of  $p$ is then defined as the set of points $x$,  such that there exists a infinite backward orbit $(x_k)_{k\leq 0}$ through $x$, i.e. $x_{k+1}=f(x_k)$ for any $k<0$ and $x=x_l$ for some $l\leq 0$, such that $x_{kT}$ goes to $p$ when $k$ goes to $-\infty$. We say that $f$ has a 
 homoclinic tangency of order $r$ at $p$ if there exists a critical point $c\in [0,1]$  flat up  to order $r$, i.e. $f(x)-f(c)=o\left( (x-c)^r \right)$,  such that $c\in W^u(\mathcal{O}(p))$ and $f^k(c)=p$ for some $k>0$.
In dimension one the stable manifold at a repelling periodic point is in general zero dimensional and a homoclinic tangency may be geometrically interpreted as a point of intersection of the stable and unstable manifold at which the graph of the interval map is tangent to the horizontal axis.   Observe that a homoclinic tangency of order $r$ is of order $s$ for any  $s\leq r$.
  
\begin{prop}\label{z}
Let $f$ be a $C^r$ interval map, $1< r<+\infty$,  with an homoclinic tangency of order $r$ at a repelling periodic point $p$. Then 

$$\limsup_{g \stackrel{C^r}{\rightarrow} f}h_{top}(g)\geq \max\left(h_{top}(f),\frac{\lambda(p)}{r}\right),$$ 
where $\lambda(p)$ is the Lyapunov exponent at $p$. \\

When moreover $\lambda(p)=R(f)$, we have then 

 $$\limsup_{g \stackrel{C^r}{\rightarrow} f}h_{top}(g)= \max\left(h_{top}(f),\frac{\lambda(p)}{r}\right).$$
\end{prop}

Note that $R(f)$ is the maximum of the Lyapunov exponents  of  all invariant measures (and zero). Indeed, firstly  the Lyapunov exponent $\lambda(\mu)$ of an $f$-invariant measure $\mu$ satisfies $\lambda(\mu)=\int \log| f'|d\mu=\int \lim_n\frac{\log |(f^n)'|}{n} d\mu\leq R(f)$. For the converse inequality we may assume $R(f)>0$. If $x_n$ is such that $|(f^n)'|(x_n)=\|(f^n)'\|_\infty>1$, we let $\nu_n$ be the atomic measure given by $\nu_n:=\frac{1}{n}\sum_{k=0}^{n-1}\delta_{f^kx_n}$. Clearly we have $\int \log | f'|d\nu_n=\frac{\log |(f^n)'|(x_n)}{n}$ and  then  any  limit $\nu$ of  $(\nu_n)_n$ in the weak-$*$ topology is $f$-invariant  and satisfies by upper semicontinuity  $\int \log |f'|d\nu\geq \lim_n\int\log |f'|d\nu_n=R(f)$.

\begin{ques}
Let $f$ be a discontinuity point of the entropy in $C^r([0,1])$ with $1< r<+\infty$. 
Does there exist $f_n\in C^r([0,1])$  with an homoclinic tangency of order $r$ at a repelling periodic point $p_{f_n}$ and going to $f$ in the $C^r$ topology  when $n$ goes to infinity? Do we have moreover
 $$\limsup_{g \stackrel{C^r}{\rightarrow} f}h_{top}(g)=\limsup_n\frac{\lambda(p_{f_n})}{r}?$$
\end{ques}

The proof of the Main Theorem is based on the study of the Buzzi-Hofbauer diagram and its behaviour under  $C^r$ perturbations.  This diagram is introduced in Section 3. In Section 2 we investigate the upper semicontinuity of the entropy of topological Markov shifts with countable state sets. The Main Theorem follows then by  applying the previous abstract framework to the Buzzi-Hofbauer diagram and  by following the strategy developed in \cite{BuRu}\cite{Bur}  (to prove existence and finiteness of measures of maximal entropy for $C^r$ interval maps $f$ with $h_{top}(f)>\frac{R(f)}{r}$).
 Finally we sketch the proof of Proposition \ref{z} in the last section, which is based on a classical construction of arbitrarily small horseshoes near a homoclinic tangency by $C^r$ perturbations.

\section{Continuity of the entropy of topological Markov shifts with countable state sets}

In this section we introduce a notion of convergence for countable Markov shifts and analyze the case of non upper semicontinuity of Gurevic entropy. We first recall standard terminology. \\

Let  $\mathcal{G}$ be an oriented graph with a countable set of vertices $V(\G)$. For $u,v\in V(\mathcal{G})$, we use the notation $u\rightarrow v$ when there is an oriented arrow from $u$ to $v$. A \textbf{closed path} at a vertex $u$  is a sequence of vertices $(u_1,...,u_{p+1})$ with $u_1=u_{p+1}=u$ and with $u_i\rightarrow u_{i+1}$ for $i=1,...,p$. The integer $p$ is called the \textbf{length} of the closed path.  The \textbf{period $p(\G)$} of the graph $\G$ is the greatest common divisor of the lengths of its closed paths. A closed path $(u_1,...,u_{p+1})$ at $u$ is said to be a \textbf{first return}  at $u$ if $u_i\neq u$ for $i\neq 1,p+1$. Any closed  path $\gamma$ at $u$ is a concatenation of first returns  $(\gamma_i)_{i=1,...,j}$ at $u$ and we denote it as follows :  \begin{eqnarray}\label{deco}\gamma:=\gamma_1*\gamma_2*...\gamma_j.\end{eqnarray} This  means that if $\gamma=(u_1,u_2,...,u_{p+1})$  there exist $1=k_1<...<k_i<...k_{j+1}=p+1$ such that 
$\gamma_i=(u_{k_i},..,u_{k_{i+1}})$ are first returns at $u$. A graph $\G$ is said to be \textbf{admissible} when for any $M\in\mathbb{N}$ and for any vertex $u\in V(\mathcal{G})$  the number of closed paths at $u$  of length $M$ is finite.\\

We  let  $\Delta_p(\mathcal{G})$ (resp. $\Delta_p^u(\mathcal{G})$) be  the set of closed paths of length $p$  in $\mathcal{G}$ (resp. at $u$). For any positive integer $M$ and for any  $u\in V(\G)$ we let $\Delta^{u}_{p,M}(\G)$ be the set of $\gamma\in \Delta^{u}_p(\G)$ such that   all first returns $\gamma_i$  at $u$ appearing in the decomposition  (\ref{deco}) of $\gamma$ in $\G$ have  length less than or equal to $M$.\\

Consider $\Sigma(\mathcal{G}):=\{(v_n)_{n}\in \mathcal{G}^{\mathbb{Z}}, \ \forall n\in \mathbb{Z},v_n\rightarrow v_{n+1} \}$. The Markov shift on  $\mathcal{G}$ is the shift $\sigma((v_n)_{n})=(v_{n+1})_{n}$ on $\Sigma(\mathcal{G})$. Obviously there is a  correspondence between $p$-periodic points of $\Sigma(\G)$ and closed paths of length $p$ in $\G$ : to any closed path $\gamma=(u_1,...,u_{p+1})$ in $\G$ of length $p$ corresponds the $p$-periodic point $\tilde{\gamma}=(v_n)_n$ of $\Sigma(\G)$ with $v_n=u_{n  (\textrm{mod}\  p) +1}$ for all $n$.

 If $F$ is a subset  of vertices of the graph $\G$ we write $[F]:=\left\{(v_n)_{n}\in \Sigma(\G) , \ v_0\in F\right\}$. We also let $\mathcal{M}(\Sigma(\G))$ be the set of invariant Borel probability measures on $\Sigma(\G)$.

\subsection{A notion of convergence}
We consider a set $E$ and a  family of subsets  $\underline{E}=(E_L)_{L \in \mathcal{L} }$  with $E=\bigcup_L E_L$.  

A family  of  oriented admissible graphs $(\G_i)_{i\in I}$ with  vertices in $E$ is said to be \textbf{uniform} with respect to $\underline{E}$ when for any $L\in\mathcal{L}$ we have
\begin{eqnarray}\label{unif}
 \sup_{i\in I }\sharp V\left(\G_i\right)\cap E_L<+\infty.
 \end{eqnarray}

\begin{defi}\label{defii}
A sequence  of graphs $(\G_n)_{n\in \mathbb{N}}$ uniform with respect to $\underline{E}$ \textbf{converges} to a graph $\G$ when $ \forall L \ \forall M \ \exists n_0 \ \forall n>n_0$ we have

\begin{center}
$ \forall u^n\in V(\G_n)\cap E_L \ \exists u\in V(\G)$ such that  
$$ \forall p, \ \sharp\Delta_{p,M}^{u^n}(\G_n)\leq \sharp\Delta_p^{u}(\G).$$
\end{center} 
\end{defi}


 Clearly the sequence of graphs $(\G_n)_{n\in \mathbb{N}}$ do not determine $\G$ uniquely. For example, if we add some edges to $\G$, then $(\G_n)_n$ is also converging to the resulting graph.  


\begin{defi}
Let $(\G_n)_{n\in \mathbb{N}}$ be a sequence of graphs uniform with respect to $\underline{E}$.   We say \textbf{ $(\xi_n)_{n\in \mathbb{N}}\in \prod_{n\in \mathbb{N}} \mathcal{M}(\Sigma(\G_n))$ goes to infinity} when for any $L\in \mathcal{L}$  we have $\lim_n\xi_n([E_L])=0$. 
\end{defi}

As   for any $L$ we have $\sup_n \sharp V(\G_n)\cap E_{L}<+\infty$ by Inequality (\ref{unif}), the sequence $(\xi_n)_{n\in \mathbb{N}}$ goes to infinity when for any $L$  the sequence   $\left(\sup_{u^n\in V(\G_n)\cap E_{L}} \xi_n([\{u^n\}])\right)_n$ goes to zero.

\subsection{Entropy and Measure of maximal entropy}
The shift space  $\Sigma(\mathcal{G})$ is  a priori not compact. Following Gurevic we define the entropy  $h(\mathcal{G})$ as the supremum of $h(\sigma,\xi)$ over all $\sigma$-invariant probability measures $\xi$. A $\sigma$-invariant measure $\xi$ is said to be maximal if $h(\sigma,\xi)=h(\mathcal{G})$. Such a measure does not always exist.  In the following we consider a  converging sequence $(\G_n)_n$ of graphs with  $\lim_n h(\G_n)>h(\G)$. We do not assume the graphs $(\G_n)_n$ or $\G$  admit maximal  measures.

We will use the following theorem of Gurevic which enables us to  work with finite connected\footnote{A graph is said to be connected when any pair of vertices may be joined by a path.}  graphs.

\begin{thm} (Corollary 1.7 \label{finit} in \cite{ggur})
Let $\G$ be an  oriented graph with  a countable set of vertices, then 

$$h(\G)=\sup_{\G_0}h(\G_0),$$

where $\G_0$ ranges  over all finite connected subgraphs of $\G$.
\end{thm}

A finite connected graph admits a unique measure of maximal entropy (the so called Parry measure). We recall now the characterization of Bowen \cite{Bow} of the Parry measure :  periodic orbits equidistribute along this measure.

\begin{thm}\label{Bowen}\cite{Bow}
Let $\G_0$ be a finite connected graph with $h(\G_0)>0$. Then  periodic points equidistribuate along the unique maximal measure $\mu$ of $(\Sigma(\G_0),\sigma)$, i.e. 

$$\frac{1}{\sharp \Delta_p(\mathcal{G}_0)}\sum_{\gamma\in \Delta_p(\mathcal{G}_0)}\delta_{\tilde{\gamma}}\xrightarrow{p, \ p(\G_0) \mid p}\mu,$$
where $\delta_{\tilde{\gamma}}$ denotes the Dirac measure at the periodic point $\tilde{\gamma}\in \Sigma(\G_0)$.

\end{thm}

Finally we recall that the Gurevic entropy of a connected oriented graph with a countable set of vertices may be written as the  exponential growth in $p$ of the number of closed paths of length $p$ at a given vertex $u$:

\begin{thm}\label{new}\cite{Gur}
Let $\G$ be a connected oriented graph with a countable set of vertices. Then for any $u\in V(\G)$, 

$$h(\G)=\lim_{p, \ p(\G) \mid p}\frac{1}{p}\log \sharp \Delta_p^u(\G).$$
\end{thm}

As the Gurevic entropy of an oriented graph $\G$ with a countable set of vertices is the supremum of the Gurevic entropy of its connected components, we always have the following  inequality for any $u\in V(\G)$:

\begin{eqnarray}\label{gugu}
h(\G)\geq\lim_{p, \ p(\G) \mid p}\frac{1}{p}\log \sharp \Delta_p^u(\G).
\end{eqnarray}

\subsection{Main proposition}

\begin{prop}\label{l}
 Let $(\G_n)_{n\in \mathbb{N}}$ be a sequence of  graphs uniform with respect to $\underline{E}$ converging  to a graph $\G$. Let $(\G'_n)_{n\in \mathbb{N}}$ be a sequence of  finite connected  graphs with $\G'_n\subset \G_n$ for all integers $n$ and   $\lim_n h(\G'_n)>h(\G)$. Let $\mu_n\in   \mathcal{M}(\Sigma(\G'_n))\subset \mathcal{M}(\Sigma(\G_n)) $  be the maximal measure of $\Sigma(\G'_n)$.  Then $(\mu_n)_{n\in \mathbb{N}}$ goes to infinity.
\end{prop}

\begin{proof}
Fix $L\in \mathcal{L}$.  Let $\epsilon>0$. Let $\delta>0$ with  $\lim_nh(\G'_n)>h(\G)+\delta$.   We fix $M$ large enough so that $1/M<\epsilon/2$ and  $\lim_pe^{-\epsilon \delta p/4}\sum_{0\leq k\leq p/M}{p\choose k}=0$ (it is well known that $\limsup_p\frac{\log{p\choose [\alpha p]}}{p}$ goes to zero when $\alpha$ goes to zero).  By convergence of $\G_n$ to $\G$ we may  fix $n_0$ large enough such that for $n>n_0$, there exists for all $u^n\in V(\G_n)\cap E_{L}$ a vertex $u\in V(\G)$ such that  for all integers $p$:
\begin{eqnarray}\label{cconv}\sharp \Delta^{u^n}_{p,M}(\G_n)\leq \sharp\Delta^u_p(\G).\end{eqnarray}
 We can also assume that  $h(\G'_n)>h(\G)+\delta$ for $n>n_0$. We fix $n>n_0$ and  we prove  now  $\mu_n([\{u^n\}])\leq\epsilon$ for any $u^n\in V(\G_n)\cap E_{L}$.

As $[\{u^n\}]$ is a clopen (possibly empty) set in $\Sigma(\G'_n)$ we have by Theorem \ref{Bowen}: 
\begin{eqnarray*}
\mu_n([\{u^n\}])&=&\lim_{p, \ p(\G'_n) \mid p}\frac{1}{\sharp\Delta_p(\G'_n)}\sum_{\gamma\in \Delta_p(\G'_n)}\delta_{\tilde{\gamma}}([\{u^n\}]);\\
 & =&\lim_{p, \ p(\G'_n) \mid p}\frac{\sharp \Delta_p^{u^n}(\G'_n)}{\sharp \Delta_p(\G'_n)}. 
 \end{eqnarray*}

The number of returns at $u^n$ in a closed path $\gamma=(u^n=u_1,u_2,...,u_{p+1}=u^n)$ of length $p
$ at $u^n$ will be denoted by $r(\gamma)$, i.e. $r(\gamma):=\{1\leq k\leq p, \ u_k=u^n\}$.  It is also the number of closed paths  at $u^n$ defining the same periodic orbit in $\Sigma(\G'_n)$. We let $\tilde{\Delta}_p^{u^n}(\G'_n)$ be the subset of closed paths $\gamma\in\Delta_p^{u^n}(\G'_n)$ such that the minimal period of  $\tilde{\gamma}\in \Sigma(\G'_n)$ is equal to $p$.  When $\gamma$ belongs to $\tilde{\Delta}^{u^n}_p(\G'_n)$ there are $p$ distinct closed paths in $\Delta_p(\G'_n)$ (in general not at the vertex $u^n$)  defining the same periodic orbit in $\Sigma(\G'_n)$. Therefore we have for any $p$:

$$\frac{\sharp \left\{\gamma \in \tilde{\Delta}_p^{u^n}(\G'_n), \ r(\gamma)<p\epsilon\right\}}{\sharp \Delta_p(\G'_n)}< \epsilon.$$

It follows immediately  from Theorem \ref{new} and $h(\G'_n)>0$ that 

$$\lim_{p, \ p(\G'_n) \mid p}\frac{\sharp \Delta_p^{u^n}(\G'_n)\setminus \tilde{\Delta}_p^{u^n}(\G'_n) }{\sharp \Delta_p(\G'_n)}\leq \lim_{p, \ p(\G'_n) \mid p}\frac{\sum_{q\mid p, \ q\neq p}\sharp \Delta_q^{u^n}(\G'_n)}{\sharp \Delta_p(\G'_n)}=0.$$

To conclude $\mu_n([\{u^n\}])\leq \epsilon$ (recall $n$ is fixed), it is enough to show that  for large $p$ the cardinality of the set  of  closed paths $\gamma$ of length $p$  at $u^n$ with $r(\gamma)\geq p\epsilon$ is less than $e^{p(h(\G'_n)-\epsilon\delta/8)}$, because  the cardinality of $\Delta_p(\G'_n)$ grows exponentially faster in $p$ by Theorem \ref{new} and thus we will get finally 

\begin{eqnarray*}
\mu_n([\{u^n\}])&=& \lim_{p, \ p(\G'_n) \mid p}\frac{\sharp \left\{\gamma \in \tilde{\Delta}_p^{u^n}(\G'_n), \ r(\gamma)<p\epsilon\right\}}{\sharp \Delta_p(\G'_n)}\leq\epsilon.
 \end{eqnarray*}

Therefore  Proposition \ref{l} will be proved once we have  shown the following claim.

\begin{Claim}
There exists $P$ (depending on $n$) such that for all $p>P$ we have 
$$\sharp \left\{\gamma \in \Delta_p^{u^n}(\G'_n), \ r(\gamma)\geq p\epsilon\right\}<  e^{p(h(\G'_n)-\epsilon\delta/8)}.$$
\end{Claim}

\begin{demof}{the Claim}
 For  $\gamma\in \Delta^{u^n}_p(\mathcal{G}'_n)$ we let  $\gamma^-$ and $\gamma^+$ be the closed paths at $u$ obtained by  concatenating  the first returns at $u^n$ appearing in $\gamma$  of length less than or equal to $M$ and larger than  $M$, respectively (recall $M$ was fixed earlier and depends only on $\epsilon$ and $\delta$). More precisely if $\gamma=
 \gamma_1*...*\gamma_j$ is the writing (\ref{deco}) of $\gamma$ in first returns at $u$, then we let 
 $\gamma^-=\gamma_{i_1}*...*\gamma_{i_k}$ where $\{i_1<...<i_k\}$ is the set of integers $i\in [1,j]$ such that the length of $\gamma_i$ is less than or equal to  $M$ and we define similarly $\gamma^+$.  For  $\gamma=(u^n=u_1,u_2,...,u_{p+1}=u^n)\in \Delta_p^{u^n}(\G'_n)$ we  also let $i(\gamma)$ be the set of integers $k\in [1, p]$ such that there exists a first return at $u_k=u^n$ of length larger than $M$. Note that the cardinality of $i(\gamma)$ is less than or equal to $p/M<p\epsilon/2$.
 
  We let $\mathcal{P}^l_p$ be the set of all subsets of $\{1,...,p\}$ whose cardinality is  less than or equal to $l$. We consider the map $\Phi:\Delta^{u^n}_p(\mathcal{G}'_n)\rightarrow \bigcup_{0\leq q\leq p}\Delta^{u^n}_{q,M}(\mathcal{G}_n)\times\Delta^{u^n}_{p-q}(\mathcal{G}'_n)\times \mathcal{P}^{p/M}_p$
which maps any $\gamma \in \Delta^{u^n}_p(\mathcal{G}'_n)$ to the triple $(\gamma^-, \gamma^+,i(\gamma))$. Clearly this map is injective. 
Now for any $\gamma\in\Delta^{u^n}_p(\mathcal{G}'_n)$  the length of $\gamma^-$ is larger than or equal to $r(\gamma)-\sharp i(\gamma)$. By Inequality (\ref{cconv}) it follows that for all $p$:

\begin{eqnarray*}
\sharp \left\{\gamma \in \Delta_p^{u^n}(\G'_n), \ r(\gamma)\geq p\epsilon\right\}& \leq & \sum_{p\geq q>p\epsilon/2}\sharp \left\{ \Delta^{u^n}_{q,M}(\mathcal{G}_n)\times\Delta^{u^n}_{p-q}(\mathcal{G}'_n)\times \mathcal{P}^{p/M}_p\right\};\\
& \leq &\sum_{\ p\geq q> p\epsilon/2 } \sharp \Delta^u_{q}(\mathcal{G}) \times \sharp \Delta^{u^n}_{p-q}(\mathcal{G}'_n)\times \sharp \mathcal{P}^{p/M}_p;
\end{eqnarray*}

and then for large $p$ (depending on the fixed $n$) we finally obtain  by Inequality (\ref{gugu}) :

\begin{eqnarray*}
\sharp \left\{\gamma \in \Delta_p^{u^n}(\G'_n), \ r(\gamma)\geq p\epsilon\right\}& \leq &  e^{\epsilon \delta p/4} \sum_{\ p\geq q> p\epsilon/2 } \sharp \Delta^u_{q}(\mathcal{G}) \times \sharp \Delta^{u^n}_{p-q}(\mathcal{G}'_n);\\
                                        &\leq & e^{\epsilon \delta p/3}\sum_{p\geq q> p\epsilon/2 }e^{qh(\G)+(p-q)h(\G'_n)};\\
                                        &< & e^{p(h(\G'_n)-\epsilon\delta/8)}.
\end{eqnarray*}
\end{demof}
As previously mentioned this concludes the proof of the  Proposition \ref{l}.
\end{proof}

\section{Proof of the Main Theorem via Buzzi-Hofbauer diagram}
\subsection{Symbolic dynamics associated to natural partitions and the Hofbauer Markov Diagram}\label{defin}
We consider a $C^1$ interval map $f$. Let $C(f)$ be the critical set of $f$, i.e. the set of vanishing points of the derivative $f'$. A (resp. strictly) monotone branch of $f$ is an open interval $I$, such that $f|_I$ is monotone (resp. strictly). We say $I$ is a \textbf{critical monotone branch} if $I$ is a monotone branch (not necessarily strictly) and the two boundary points of $I$  belong to $C(f)\cup \{0,1\}$. A (countable) collection $P$ of disjoint critical monotone branches is called a \textbf{natural partition} of $f$ when the union of all monotone branches in $P$ covers any strictly monotone branch of $f$, i.e.  $\{f'\neq 0\}\subset \bigcup_{I\in P} I$.

 For a natural partition $P$ of $f$ the two-sided \textbf{symbolic dynamic}  $\Sigma(f,P)$ associated to $f$ is defined as the shift on the closure in $P^{\mathbb{Z}}$ (for the product topology) of the two-sided sequences $A=(A_n)_n$ such that for all $n\in \mathbb{Z}$ and $l\in \mathbb{N}$ the word $A_n...A_{n+l}$ is \textbf{admissible}, by which we mean  that  the intersection $\bigcap_{k=0}^lf^{-k}A_{n+k}$ is non empty and the $f^{l+1}$-image of this open interval is not reduced to a point.
  The \textbf{follower set} of a finite $P$-word $B_n...B_{n+l}$  is $fol(B_n...B_{n+l}):=\{A_{n+l}A_{n+l+1}...\in P^{\mathbb{N}},$ s.t. $\exists (A_n)\in \Sigma(f,P)$  with $A_n...A_{n+l}=B_n...B_{n+l}\}$.

Let  $\mathcal{P}$ be the set of admissible  $P$-words. We consider the following equivalence relation on $\mathcal{P}$. We say $A_{-n}...A_0\sim B_{-m}...B_0$ if and only if  there exist $0\leq k\leq \min(m,n)$ such that :

\begin{itemize}
\item $A_{-k}...A_0 = B_{-k}...B_0$;
\item $fol(A_{-n}...A_0)=fol(A_{-k}...A_0)$;
\item   $fol(B_{-m}...A_0)=fol(B_{-k}...A_0)$.
\end{itemize}

We endow the quotient space $\mathcal{D}=\mathcal{D}(f,P):=\mathcal{P}/\sim$   with a structure of oriented graph, known as the Buzzi-Hofbauer diagram, in the following way \cite{Buz1} : there exists an oriented arrow  $\alpha \rightarrow \beta$ between two elements $\alpha,\beta$ of $\mathcal{D}$ if and only if there exists an integer $n$ and $A_{-n}...A_0A_1\in \mathcal{P}$ such that $\alpha\sim A_{-n}...A_0$ and $\beta\sim A_{-n}...A_0A_1$. 

  \textbf{The significative part} of $\alpha\in\mathcal{D}$ is the representative  $A_{-n_{\alpha}}...A_0$ of $\alpha$ with the shortest length.  Such a word $A_{-n_{\alpha}}...A_0$ is also said \textbf{irreducible}, it is the shortest element in its class. In particular $fol(A_{-n_{\alpha}}...A_0)\neq fol(A_{-n_{\alpha}+1}...A_0)$ when $n_\alpha>0$. 
 
We let $\mathcal{D}_N$ be  the subset of $\mathcal{D}$ generated by elements of $\bigcup_{k=1}^NP^k$, i.e. $\alpha\in\mathcal{D}_N$  if and only if there exists $0\leq k< N$ and $A_{-k}...A_0\in \mathcal{P}$ such that  $\alpha\sim A_{-k}...A_0$. Thus $\mathcal{D}_N$ is  the subset of $\mathcal{D}$ whose significative part has length less than or equal to $N$.

\subsection{Convergence of the Buzzi-Hofbauer diagram}\label{contgr}

We let $E$ be the union $\bigcup_{f,P}\mathcal{D}(f,P)$ over all $C^1$ interval maps $f$ and all natural partitions $P$ of $f$. For any $\alpha\in \mathcal{D}(f,P)$ we let $L(\alpha)$ be  the length of $f^{N'+1}(\bigcap_{0\leq i\leq N'} f^{i-N'}A_{-i})$  for some (any) representative $A_{-N'}...A_0\in \mathcal{P}$ of $\alpha$.  For any $(N,K)\in (\mathbb{N}\setminus \{0\})^2$ we consider the subset $E_{N,K}$ of $E$ defined  by 
$$E_{N,K}:=\{\alpha\in \bigcup_{f,P}\mathcal{D}_N(f,P), \ L(\alpha)\geq 1/K\}.$$ Clearly we have $E=\bigcup_{N,K}E_{N,K}$. \\

We analyse  now the convergence of the Buzzi-Hofbauer diagrams associated to a converging sequence of $C^1$ interval maps. We begin with some preliminary facts. \\

 \textbf{Fact 0:} Let $\mathcal{F}$ be a $C^1$ bounded set of $C^1$ interval maps. For any $N,K,M$ there exists $\tilde{N},\tilde{K}$ depending only on $N,K,M,\sup_{f\in \mathcal{F}}\|f'\|_{\infty}$ such that for any $f\in \mathcal{F}$ and for any natural partition $P$ of $f$, any closed path in $\mathcal{D}(f,P)$ of length $M$ at  a vertex in $E_{N,K}$ is contained in $E_{\tilde{N},\tilde{K}}$.\\
 
  \begin{proof}
The proof follows immediately from the two following properties of the Buzzi-Hofbauer diagram:
  \begin{itemize}
\item if $\alpha\in \mathcal{D}_N(f,P)$  and  $\alpha\rightarrow \beta$ then $\beta\in \mathcal{D}_{N+1}(f,P)$; 
\item  if $\alpha\rightarrow \beta$ then $L(\beta)\leq \|f'\|_{\infty} L(\alpha)$.
\end{itemize}
Indeed one only needs to put $\tilde{N}=N+M$ and $\tilde{K}=[K\sup_{f\in \mathcal{F}}\|f'\|_{\infty}^M]+1$.
\end{proof}
 
 For any $C^1$ interval map $f$, any natural partition $P$  of $f$  and any positive integer $m$, we denote by $P(m)$ the set of elements of $P$ where $|f'|$ attains $1/m$. \\
 
\textbf{Fact 1:} Let $\mathcal{F}$ be a $C^1$ bounded set of $C^1$ interval maps. 
 For any $N,K$ there exists $m$ depending only on $N,K,\sup_{f\in \mathcal{F}}\|f'\|_{\infty}$ such that for any $f\in \mathcal{F}$ and any natural partition $P$ of $f$, the set $\mathcal{D}(f,P)\cap E_{N,K}$ is generated by $\bigcup_{k=1}^N P(m)^k$. 
 
 \begin{proof}
Let $\alpha\sim  A_{-N'}...A_0$ with $N'<N$ and  with the length of $f^{N'+1}(\bigcap_{0\leq i\leq N'} f^{i-N'}A_{-i})$  larger than or equal to $1/K$. 
Clearly we have $\prod_{i=0}^{N'}\sup_{x\in A_{-i}}|f'(x)| \geq 1/K$ and thus we have for any $0\leq i \leq N'$: 

$$\sup_{x\in A_{-i}}|f'(x)|\geq \frac{1}{K\|f'\|_{\infty}^{N'}}.$$
Therefore it is enough to take $m>K\max(1,\sup_{f\in \mathcal{F}}\|f'\|_{\infty})^N$.
 \end{proof}

 From now we fix  $C^1$ interval maps $f,(f_n)_{n\in \mathbb{N}}$ such that $(f_n)_n$ converges  to $f$ in the $C^1$ topology and we consider natural partitions $(P_n)_n$ of $(f_n)_n$. \\

\textbf{Fact 2:} For any positive integer $m$ we have 
$$\sup_n \sharp P_n(m)<+\infty.$$

\begin{proof}
Let $m\in \mathbb{N}$.  Assume  $\sup_n\sharp P_n(m)=+\infty$. Clearly for any fixed $n$ the set $P_n(m)$ is finite.  Then up to extract a subsequence one may find for all $n$ an element $(a_n,b_n)$ of $P_n(m)$ whose length $|a_n-b_n|$ goes to zero when $n$ goes to infinity. For any $n$ we let $c_n\in (a_n,b_n)$ with $|f'_n(c_n)|\geq1/m$. We may assume that  $a_n,b_n$ are not boundary  points of the unit interval, so that $f_n'(a_n)=f_n'(b_n)=0$.  Any  accumulation  point $(a,b,c)$ of the sequence $(a_n,b_n,c_n)$ satisfies $a=b=c$.    Moreover as $(f_n)_n$ is $C^1$ converging to $f$, we have $f'(a)=f'(b)=0$  and $|f'(c)|\geq 1/m$ and we get a contradiction. Therefore  we have $\sup_n\sharp P_n(m)<+\infty$ for any given $m$. 
\end{proof}

When $A$ is an open interval and $n$ is an integer we denote by $A^n$ the element\footnote{if it exists - when using the notation $A^n$ we claim it is well defined.} of $P_n$ which contains the midpoint of $A$.\\

\textbf{Fact 3:}\label{part} There exists a subsequence $(n_k)_k$ with the following property. There is a nondecreasing sequence $(R_m)_m$ of finite collections of disjoint critical monotone branches of $f$, such that for any positive integer $m$:
\begin{itemize}
\item for all $k\geq m$ we have $P_{n_k}(m)=\{A^{n_k}, \ A \in R_m \}$;
\item for any $A\in R_m$ the sequence $(A^{n_k})_{k\geq m}$ goes to $A$ in the Hausdorff topology   when $k$ goes to infinity.
\end{itemize} 
Moreover  the union $R:=\bigcup_{m\in \mathbb{N}}R_m$ defines a natural partition of  $f$.

\begin{proof}
 By a Cantor's diagonal argument, one may find  a subsequence $(n_k)_k$ 
such that for any $m$ the cardinality of $(P_{n_k}(m))_k$ is constant for $k\geq m$  and that $P_{n_k}(m)$ converges in the 
Hausdorff topology to a (finite) collection $R_m$ of  critical monotone branches of $f$. We may also assume $P_{n_k}(m)$ so close  to $R_m$ for any $k\geq m$  that   $P_{n_k}(m)$ is given by the family $\{A^{n_k}, \ A \in R_m \}$.

 Finally observe that for any point $x$ in $\{f'\neq 0\}$ there exist an integer $m$ and an open neighborhood $U$ of $x$ which is contained in a unique  element of $P_{n_k}(m)$ for large $k$. By taking the limit in $k$ we get that $U$ is contained in some $A\in R_m\subset P$. Therefore $R=\bigcup_{m\in \mathbb{N}}R_m$ is a natural partition of $f$.
\end{proof}

We let $\mathcal{D}^{n}$ and  $\mathcal{D}$ be the Buzzi-Hofbauer diagram associated to the natural partitions  $P_{n}$ and $R$ (given by the  above Fact 3) of $f_{n}$ and $f$ respectively. It follows immediately from Fact 1 and Fact 2 that $(\mathcal{D}^{n})_n$ is uniform with respect to $\underline{E}=(E_{N,K})_{N,K}$. From Fact 0 and Fact 1 one also easily sees that all these diagrams are admissible graphs. In fact by Theorem \ref{new} admissibility may be deduced  more directly from the finiteness of  Gurevic entropy which will follow from the Isomorsphism Theorem (Theorem \ref{encor}).

 We prove now the convergence of the Buzzi-Hofbauer diagrams by taking  again a subsequence.

\begin{lem}\label{ppp}
There exists a subsequence $(n_{k_l})_l$  such that 
the Buzzi-Hofbauer diagrams $(\mathcal{D}^{n_{k_l}})_l$  of  $(f_{n_{k_l}})_l$  converge to the Buzzi-Hofbauer diagram $\mathcal{D}$ of $f$. 
\end{lem}

\begin{proof}
Let  $N,K,M$ be positive integers.  By Facts 0,1,3,  there exists an integer $m$ depending only on $N,K,M,\sup_n \|f_n'\|_\infty$ such that for $k\geq m$ any closed path $\gamma_k$ of length $M'\leq M$  at a vertex  $ \alpha^{n_k}\in\mathcal{D}^{n_k}\cap E_{N,K}$ is given  by a
 $P_{n_k}$-word $A^{n_k}_{-N'}...A_0^{n_k} A_1^{n_k} ....A_{M'}^{n_k} $ with $N'<N$ such that \begin{eqnarray*}\label{eqq}
\alpha^{n_k} \sim A^{n_k}_{-N'}...A_0^{n_k}\sim  A^{n_k}_{-N'}...A_0^{n_k}A_1^{n_k} ...A_{M'}^{n_k},
 \end{eqnarray*}
  where $A_{-N'},...,A_{M'} $ belongs to $R_m$. Up to extract a subsequence one may assume by uniform convergence of $(f_{n_k})_k$ to $f$ that  $  A_{-N'}...A_0$ is an admissible word  and moreover   $A_{-N'}...A_0 A_1...A_{M'} \sim A_{-N'}...A_0$. This last relation defines a closed path $\gamma$ of length $M'$ at the class $\alpha$ of   $ A_{-N'}...A_0$ in $\mathcal{D}$. The function $\phi_k:\Delta^{\alpha^{n_k}}
_{M}(\mathcal{D}^{n_k})\rightarrow \Delta^{\alpha}
_{M}(\mathcal{D})$ mapping $\gamma_k$ to $\gamma$ may be extended to $\Delta^{\alpha^{n_k}}
_{p,M}(\mathcal{D}^{n_k})$ for all $p$ by concatening the $\phi_k$ images of the first returns at $\alpha^{n_k}$ appearing in the decomposition of a closed path in $\Delta^{\alpha^{n_k}}
_{p,M}(\mathcal{D}^{n_k})$. The resulting map, which  takes value in $\Delta^{\alpha}_{p}(\mathcal{D})$,  is  injective. Indeed,   any closed path $\gamma_k$ in $\Delta^{\alpha^{n_k}}
_{p,M}(\mathcal{D}^{n_k})$ is given  by a
 $P_{n_k}$-word $A^{n_k}_{-N'}...A_0^{n_k} A_1^{n_k} ....A_{p}^{n_k} $ with $N'<N$ and with $A_{-N'},...,A_{p} \in R_m$ such that \begin{eqnarray*}\label{eqq}
\alpha^{n_k} \sim A^{n_k}_{-N'}...A_0^{n_k}\sim  A^{n_k}_{-N'}...A_0^{n_k}A_1^{n_k} ...A_{p}^{n_k}.
\end{eqnarray*} 
Then $\phi_k(\gamma_k)$ is the closed path in $\mathcal{D}$ of length $p$ at $\alpha$ whose $q^{th}$ term is  the class of $A_{-N'}...A_0A_1 ...A_{q-1}$ in $\mathcal{D}$ for any $1\leq q\leq p$. Since these classes allow us to determine $A_q$ for $1\leq q\leq p$ two closed paths in $\Delta^{\alpha^{n_k}}
_{p,M}$ with the same image under $\phi_k$ coincide. It follows that 
$$ \sharp \Delta^{\alpha^{n_k}}
_{p,M}(\mathcal{D}^{n_k})\leq \sharp \Delta^{\alpha}
_{p}(\mathcal{D}).$$

Finally by using a Cantor's diagonal argument we may extract a subsequence $(n_{k_l})_l$ such that this holds true for any $N,K,M$ whenever $l$ is large enough, i.e. the Buzzi-Hofbauer diagrams $(\mathcal{D}^{n_{k_l}})_l$  of  $(f_{n_{k_l}})_l$  converge to the Buzzi-Hofbauer diagram $\mathcal{D}$ of $f$.

\end{proof}

\subsection{Isomorphism Theorem}\label{isom}

Let $f$ be a $C^1$ interval map and $P$ be a natural partition of $f$.
The symbolic dynamics  extends the dynamics on the interval as follows. For any  $A=(A_n)_{n\in \mathbb{Z}}\in \Sigma(f,P)$ we let $\pi_0(A):=\bigcap_{k\in\mathbb{N}}\overline{\bigcap_{l=0}^kf^{-l}A_l}$. Since $f$ is monotone on each element of $P$ the set $\bigcap_{l=0}^kf^{-l}A_l$ is an interval for all $k\in\mathbb{N}$. In particular  $\pi_0(A)$ is a point or  a compact non trivial interval ; but this last possibility occurs only for a countable set of elements $(A_n)_{n\in \mathbb{N}}$. Therefore there is a Borel   subset of $\Sigma(f,P)$ with full $\mu$-measure for any $\sigma$-invariant measure $\mu$ with positive entropy such that  the restriction of  $\pi_0$ to this subset defines  a Borel map (in the following this Borel map will be also denoted by $\pi_0$).

  We also consider the projection $\pi_1:\Sigma(\mathcal{D})\rightarrow \Sigma(f,P)$ defined by $\pi_1((\alpha_n)_n)=B_n$ where $B_n$ is the last letter of the word $\alpha_n$. 
  
  We recall now the Isomorphism Theorem for $C^{1+\alpha}$ interval maps obtained in \cite{Bur} based on previous works of J.Buzzi \cite{Buz1} :

\begin{thm}\cite{Bur}\label{encor}
Let $f$ be a $C^{r}$ map of the interval with $r>1$.
The Borel map $\pi:=\pi_0\circ\pi_1:\Sigma(\mathcal{D})\rightarrow [0,1]$ induces a bijection, preserving the entropy, between ergodic invariant measures  with positive entropy of  $(\Sigma(\mathcal{D}),\sigma)$ and $([0,1],f)$.
\end{thm}

\begin{rem}
In \cite{Bur}, \cite{Buz1} the authors consider  the particular natural partition given by the set of connected components of  $[0,1]\setminus D(f)$, where  $D(f)$ is the set of points which do not belong to the interior of any strictly monotone branch of $f$. However the proof  of the Isomorphism Theorem applies straightforwardly to general natural partitions. Here we need to work with general natural partitions because  Lemma \ref{ppp} is false for the above particular choice used in \cite{Bur}, \cite{Buz1}. 
\end{rem}

\subsection{Bound of the entropy at infinity}

We first prove   the following  refinement of Proposition 3 of \cite{BuRu} :
\begin{lem}\label{adapt}
Let $(f_n)_{n\in \mathbb{N}}$ be  a sequence of $C^1$ interval maps converging in the $C^1$ topology.  For any integer $n$ we let $\mathcal{D}^n$ be the Buzzi-Hofbauer diagram associated to some natural partition $P_n$ of $f_n$.

Let  $(\xi_n)_{n}$ be a sequence of ergodic $\sigma$-invariant  measures on  $\Sigma(\mathcal{D}^n)$ such that :
\begin{itemize}
\item $h(\sigma,\xi_n)>0$ for all $n$;
\item  $( \xi_n)_n$ goes to infinity.
\end{itemize}

Then for all  $N\in\mathbb{N}$ we have $\lim_{n}\xi_n([\mathcal{D}^n_N])=0$.
\end{lem}

\begin{proof}
 For any  $N,n,m$ we denote by 
$\mathcal{D}^n_{N,m}$ the subset of $\mathcal{D}_N^n$ given  by irreducible $P_n$-words $A_{-N'}...A_{0}$ (note that $N'<N$) with $$\sup_{x\in \bigcap_{i=0}^{N'} f_n^{-i}A_{-N'+i}}|(f_n^{N'+1})'(x)|>1/m.$$ Obviously $\mathcal{D}^n_{N,m}\subset \mathcal{D}^n_{N,m'}$ for $m'>m$. We fix now $N$. \\

\underline{\textit{Step 1:}} $\xi_n([\mathcal{D}_N^n\setminus\mathcal{D}^n_{N,m}])\xrightarrow{m\rightarrow +\infty}0$ uniformly in $n$.\\
Assume by contradiction that for infinitely many $m$ there exists $n=n(m)$ with  $\xi_n([\mathcal{D}_N^n\setminus\mathcal{D}^n_{N,m}])>a>0$. 
Let $\mu_n=\pi^* \xi_n$ the induced $f_n$-invariant ergodic measure on $[0,1]$ and $\lambda_{\mu_n}$ its Lyapunov exponent. By the ergodic thereom one easily  gets that 
$$\lambda_{\mu_n}\leq \log^+\|f_n'\|_\infty-\frac{a\log m}{2N}.$$
Therefore by choosing $m > \max(1,\sup_{n}\|f_n'\|_{\infty})^{2N/a}$ we conclude that the Lyapunov exponent of   $\mu_n$ is negative and thus its entropy is zero by Ruelle's Inequality, $h(f,\mu_n)\leq \max(\lambda_{\mu_n},0)$. We get a contradiction with our first assumption $h(\sigma,\xi_n)>0$ as we have by the Isomorphism Theorem (Theorem \ref{isom}),  $h(f,\mu_n)=h(\sigma,\xi_n)$.  \\

\underline{\textit{Step 2:}} for all $m$ there exists $K$ with $\bigcup_n\mathcal{D}^n_{N,m}\subset E_{N,K}$.\\ Observe that
for any irreducible $P_n$ word $A_{-N'}...A_{0}$  representing $\alpha\in \mathcal{D}^n_{N,m}$ there exists  $z\in \bigcap_{i=0}^{N'} f_n^{-i}A_{-N'+i}$ with  $|
(f_n^{N'+1})'(z)|>1/m$. There is a  positive integer $K'$   such that  $|(f_n^{N'+1})'(t)|>1/2m$ for any $t$ in the $1/K'$-neighborhood of $z$ in $[0,1]$ (in particular this neighborhood  is entirely contained in $\bigcap_{i=0}^{N'} f_n^{-i}A_{-N'+i}$).  We may also choose $K'$ independent of $z$ by uniform continuity of $f_n$ but also  of $n$ by $C^1$ convergence of $ (f_n)_n$. Consequently $L(\alpha)\geq \frac{1}{2mK'}$ and thus $\alpha$ belongs to $E_{N,K}$ for $K=2mK'$.\\

\underline{\textit{Step 3:}} Conclusion.\\
Let $\epsilon>0$. By Step 1 we may find $m$ so large that for all $n$ we have  $\xi_n([\mathcal{D}_N^n\setminus\mathcal{D}^n_{N,m}])<\epsilon/2$. Then there exists by Step 2 an integer $K=K(m)$ such that  for all $n$ we have $\mathcal{D}^n_{N,m}\subset E_{N,K}$ and therefore we obtain for all $n$:
$$\xi_n([\mathcal{D}_N^n\setminus E_{N,K}]) \leq  \xi_n([\mathcal{D}_N^n\setminus\mathcal{D}^n_{N,m}])<\epsilon/2.$$
Finally, as  $(\xi_n)_n$ goes to infinity,  there exists $n_0$ such that for all $n>n_0$ we have 
  $$\xi_n([ E_{N,K}])<\epsilon/2$$ and thus 
   $$\xi_n([\mathcal{D}_N^n])\leq \xi_n([\mathcal{D}_N^n\setminus E_{N,K}]) +\xi_n([ E_{N,K}])<\epsilon.$$
\end{proof}

The entropy of a sequence of ergodic $\sigma$-invariant measures going to infinity may be bounded from above as follows.


\begin{prop}\label{pf}
Let $(f_n)_{n\in \mathbb{N}}$ be  a sequence of $C^1$ interval maps converging in the $C^1$ topology with $\sup_n\|(f_n)^{(r)}\|_\infty<\infty$. For any integer $n$ we let $\mathcal{D}^n$ be the Buzzi-Hofbauer diagram associated to some natural partition $P_n$ of $f_n$.

Let  $(\xi_n)_n$ be a  sequence of ergodic $\sigma$-invariant measures on $\Sigma(\mathcal{D}^n)$ such that $\lim_{n}\xi_n([\mathcal{D}^n_N])=0$ for all  $N\in\mathbb{N}$. 

Then  we have for any weak limit $\mu:=\lim_k\mu_{n_k}$ of $(\mu_n)_n:=(\pi^*\xi_n)_n$ :

 $$\limsup_{k} h(f_{n_k},\mu_{n_k})\leq \frac{\int\log^+ |f'|d\mu}{r}.$$
\end{prop}

Proposition \ref{pf} is shown by  following straightforwardly the proof of Proposition 4 in \cite{Bur}. In \cite{Bur} the statement concerns a single map, but all the proof works when considering a $C^1$-converging sequence  of $C^r$ maps with uniformly bounded $r$-derivative. Let us sketch the main ideas of this generalization.

For any $A\in \Sigma(f_n,P_n)$ and any $N\in \mathbb{N}$ we let $r_N(A)$ be the least integer $m\in \mathbb{N}\cup\{+\infty\}$ with $m>N$ such that $fol(A_{-m-1}...A_0)\neq fol(A_{-m}...A_0)$. This last inequality implies that there exist a point $y\in \partial A_{-m-1}$ which shadows the piece of orbit $A_{-m}...A_0$, i.e.  $f_n(y)\in \bigcap_{l=0}^n f_n^{-l}A_{-m+l}$.

Then, a sequence $(\nu_n)_n\in \prod_m\mathcal{M}(\Sigma(f_n,P_n),\sigma)$ is said to satisfy the \textbf{shadowing property} when for all integers $N$ we have $\lim_n\nu_n(r_N<+\infty)=1$. It follows easily from the definitions that if   $(\xi_n)_{n}$ is a sequence of $\sigma$-invariant ergodic measures on  $\Sigma(\mathcal{D}^n)$ with  $\lim_{n}\xi_n([\mathcal{D}^n_N])=0$, then the sequence $(\pi_1^*\xi_n)_n$ satisfy the shadowing property. 

We consider now a sequence $(\nu_n)_n\in \prod_m\mathcal{M}(\Sigma(f_n,P_n),\sigma)$ satisfying the shadowing property and we explain how to bound the entropy of $(\mu_n)_n:=(\pi_0^*\nu_n)_n$.  For any $N$, a typical orbit of length $L$ for $\mu_n$  with $n$ large enough may be shadowed at almost any time by at most $k\leq \frac{L}{N}$ critical points of $f_n$ on disjoint orbit segments  of length larger than $N$. Then by using combinatorial arguments  one may bound the entropy of $\mu_n$  by the exponential growth rate in $L$ of the number of such  possible $k$-uples of critical points. This last rate may be bounded in terms   of the Lyapunov exponent of $\mu_n$ as follows.  If at a starting time of an orbit segment the $|f'_n|$-value of our $\mu_n$-typical orbit is  $l$ then the number of possible shadowing points for this piece is less than $C\|(f_n)^{(r)}\|_\infty l^{\frac{1}{r-1}}$ with $C=C(r)$ depending only on $r$. Indeed for a $C^r$ interval map $g$ the number of  critical strictly monotone branches where $|g'|$ exceeds the value $l$ is at most $C(r)\|g^{(r)}\|_\infty l^{\frac{1}{r-1}}$, see Lemma 4.1 of \cite{DB}. Then  as $\sup_n\|(f_n)^{(r)}\|_\infty<+\infty$ there exists for any $\epsilon>0$ an integer $n_\epsilon$ such that for $n>n_\epsilon$ we have 
$h(f_n,\mu_n)\leq \frac{\int -\log^- |f_n'|d\mu_n}{r-1}+\epsilon$. Together with Ruelle inequality, $h(f_n,\mu_n)\leq \max(\int \log |f_n'|d\mu_n,0)$, we get for $n>n_\epsilon$ (we may assume $h(\mu_n)>0$)

\begin{eqnarray*}
h(f_n,\mu_n) & \leq & \frac{\int -\log |f_n'|d\mu_n+\int \log^+ |f_n'|d\mu_n}{r-1}+\epsilon;\\
& \leq & \frac{-h(f_n,\mu_n)+\int \log^+ |f_n'|d\mu_n}{r-1}+\epsilon.
\end{eqnarray*}
and thus for $n>n_\epsilon$ we get $$h(f_n,\mu_n) \leq  \frac{\int \log^+ |f_n'|d\mu_n}{r}+\epsilon.$$
We conclude by continuity of the integral in the right member that for any limit  $\mu:=\lim_k\mu_{n_k}$ of $(\mu_n)_n$:
$$\limsup_{k} h(f_{n_k},\mu_{n_k})\leq \frac{\int\log^+ |f'|d\mu}{r}.$$

\subsection{Proof of the Main Theorem for interval maps}
We can now prove the Main theorem for interval maps. Let $(f_n)_n$ be a sequence of $C^r$ interval maps converging in the $C^1$ topology to $f$   with $\sup_n\|(f_n)^{(r)}\|_\infty<+\infty$ such that $\limsup_n h_{top}(f_n)>h_{top}(f)$. We may assume that $\lim_nh_{top}(f_n)$ exists and is finite since we have $\limsup_n h_{top}(f_n)\leq \limsup_nR(f_n)\leq R(f)$. Let $(P_n)_n$ be natural partitions of $(f_n)_n$. By Lemma \ref{part} and Lemma \ref{ppp} we can suppose by taking  a subsequence that there exists a natural partition $P$ of $f$ such that the Buzzi-Hofbauer diagram $\mathcal{D}^n$ associated to $(f_n,P_n)$ converge to the Buzzi-Hofbauer diagram  $\mathcal{D}$ associated to $(f,P)$. According to the Isomorphism Theorem (Theorem \ref{encor}) we have $h(\mathcal{D}^n)=h_{top}(f_n)$ and  $h(\mathcal{D})=h_{top}(f)$. Also by Theorem \ref{finit} one may find finite connected subgraphs $\G_n\subset \mathcal{D}^n$ with   $\lim_nh(\mathcal{D}^n)=\lim_nh(\G_n)>h(\mathcal{D})$. The Main Proposition claims then that the measure $\xi_n$ of maximal entropy of $\G_n$ goes to infinity when $n$ goes to infinity. Then  by Lemma \ref{adapt}  we get $\lim_{n}\xi_n([\mathcal{D}^n_N])=0$ for all $N$ and thus by  Proposition \ref{pf} we have for any weak limit $\mu\in \mathcal{M}(f,[0,1])$ of $(\mu_n)_n=(\pi^*\xi_n)$:

\begin{eqnarray*}
\lim_n h_{top}(f_n)&=&\lim_n h(\mathcal{D}^n);\\
&=&\lim_n h(\G_n);\\
&=& \lim_n h(\xi_n);\\
&=&\lim_n h(\mu_n);\\
&\leq & \frac{\int\log^+ |f'|d\mu}{r};\\
&\leq & \frac{\log^+ \|f'\|_\infty }{r}.
\end{eqnarray*}

Now if $m$ is an integer so large that $R(f)\simeq \frac{\log^+\|(f^m)'\|_\infty}{m}$, we may apply the previous result to the sequence $f_n^m$ and to $f^m$ since we have
\begin{eqnarray*}
\lim_n h_{top}(f_n^m)&=&m\lim_nh_{top}(f_n);\\
&> & mh_{top}(f);\\
&=&h_{top}(f^m).
\end{eqnarray*} 

Then we get 

\begin{eqnarray*}
\lim_n h_{top}(f_n)&=&\lim_n\frac{h_{top}(f_n^m)}{m};\\
&\leq & \frac{\log^+ \|(f^m)'\|_\infty }{mr};\\
& \lesssim & \frac{R(f)}{r}.
\end{eqnarray*}

This concludes the proof of the Main Theorem for interval maps.



\subsection{Proof of the Main Theorem for circle maps}
For circle maps the theorem is proved by reduction to the case of interval maps as follows. In the  assumptions of the Main 
Theorem one may assume that $f$ has positive entropy. Indeed when $f$ has zero entropy the inequality in the Main 
Theorem is just given by Yomdin's inequality (\ref{yomd}). It is well known that a circle map with positive entropy has an 
horseshoe \cite{Misho}, in particular there exist a positive integer $k$ and  an interval $J\neq \mathbb{S}^1$ such that the 
closure of $J$ is contained in the $f^k$-image of the interior of $J$. This still holds for maps $g$ which are $C^0$-close to $f$. For $n$ large enough we  let $p_n\in J$  be a $k$-periodic point of $f_n$. By extracting a subsequence we may 
assume that $(p_n)_n$ is converging to a $k$-periodic point $p$ of $f$. Then the interval  maps $\tilde{f^p}$ and $\tilde{f_n^p}$, obtained from $f^p$ and $f_n^p$ by blowing up the circle at the fixed point  $p$ and $p_n$ respectively, 
are $C^r$ interval maps such that  $(\tilde{f_n^p})_n$ is converging  to $\tilde{f^p}$ in the $C^r$ topology. By 
applying the Main Theorem for interval maps we get

$$\limsup_n h_{top}(\tilde{f_n^p})\leq \max\left(h_{top}(\tilde{f^p}), \frac{R(\tilde{f^p})}{r}\right)$$

It is easily checked that  \begin{eqnarray*}
 h_{top}(\tilde{f_n^p})  =  h_{top}(f_n^p) & = & ph_{top}(f_n); \\
 h_{top}(\tilde{f^p})  =  h_{top}(f^p) & = & ph_{top}(f);\\
 R(\tilde{f^p})  = R(f^p) & = & pR(f);
 \end{eqnarray*} so that 
 
 $$\limsup_n h_{top}(f_n)\leq \max\left(h_{top}(f), \frac{R(f)}{r}\right).$$

\section{Proof of Proposition \ref{z}}
\subsection{Discontinuity at maps with homoclinic tangencies of order $r$}

We sketch now the proof of Proposition \ref{z}. Similar examples have been already built in \cite{Buz0} \cite{Rue} \cite{bure} and we refer to this paper for details on the construction. 
Let $f$ be a $C^r$ interval map with an homoclinic tangency of order $r$ at a repelling fixed point $p$. We denote by $c$ the critical point flat up to order $r$ in the unstable manifold of $p$ with $f^k(c)=p$ for some $k>0$. We perturb $f$ only on a small neighborhood 
$]c-\delta,c+\delta[$  by letting the perturbation $g$ be a sinusoidal of the form $g(x)=a\sin(Nx/\delta)+f(c)$. We may choose  $a=C\delta |f'(p)|^{-l} $ for some constant $C$ to get a $N$-horseshoe for  $f^l$ with $l\gg |\log \delta|$. The entropy of this horseshoe is given by $\frac{\log N}{l}$ whereas to ensure the $C^r$ closeness one may  take  $aN^r=\delta^r/l$. In this way we obtain $g_l$ going   to $f$ in the $C^r$ topology such that  

\begin{eqnarray*}
h_{top}(g_l)& \geq & \frac{ \log N}{l};\\
&\geq & \frac{\log(\delta^r/al)}{rl};\\
&\geq &  \frac{\log\left(\delta^{r-1}|f'(p)|^l/Cl\right)}{rl};\\
&\geq & \frac{\lambda(p)}{r}+o(1/l).
\end{eqnarray*}

By lower semicontinuity of the entropy we finally get:

$$\limsup_l h_{top}(g_l)\geq \max\left(h_{top}(f),\frac{\lambda(p)}{r}\right).$$

\begin{appendix}
\section*{Appendix}

As a consequence of Misiurewicz's result (\ref{mi95}) we give here a short proof of the following theorem.

\begin{thmm}\cite{Iw}
For any real number $r>1$, the entropy is  continuous on the set $D^r([0,1])$ of $C^r$ interval maps with no critical point flat up to order $r$.  
\end{thmm}

\begin{proof}
If $f$ belongs to $ D^r([0,1])$ then its critical set is finite and thus $f$ is piecewise monotone. In fact for any $C^r$ bounded set $\mathcal{V}$  there exist an integer $k$ and a $C^1$ neighborhood $\mathcal{U}$ of $f$  such that any $g\in \mathcal{U}\cap \mathcal{V}$ belongs to $\mathcal{M}_k([0,1])$: 

\begin{Claim} \nonumber Let $r>1$, $R>0$ and  $f\in D^r([0,1])$. There exist $\epsilon_0>0$ and a $C^1$ neighborhood $\mathcal{U}$ of $f$ such that any map $g\in 
\mathcal{U}\cap C^r([0,1])$ with  $\|g^{(r)}\|_\infty\leq R$ has at most $r-1$ critical points in any ball of radius $\epsilon_0$, in particular $g$ is in $\mathcal{M}_k^1([0,1])$ with $k=
[r/\epsilon_0]+1$.
\end{Claim}

\begin{proof}[Proof of the Claim]
 Arguing by contradiction we assume that for some $r>1$, $R>0$, there exist $f\in D^r([0,1])$ such that for any $\delta>0$ and for any $\epsilon>0$  there 
 exists a $C^r$ map $g_{\delta,\epsilon}$ which is   $\delta$ $C^1$-close to $f$   with $\|(g_{\delta,\epsilon})^{(r)}\|_\infty<R$ and with $r$ critical points  of 
 $g_{\delta,\epsilon}$ in an interval $I_{\delta,\epsilon}$ of length $\epsilon$. Let $x$ be an accumulation point of the intervals $(I_{\delta,\epsilon})_{\delta,\epsilon}$ when $\delta$ and $\epsilon$ go to zero. 
 In particular for some arbitrarily small $\delta$ and $\epsilon$ the length of $g_{\delta,\epsilon}(I)$ for any $ I\supset I_{\delta,\epsilon}$ has length less than $R|I|
^r$ by Lemma 3.2 of \cite{BLY}.  Thus  as $(g_{\delta,\epsilon})_\delta$ converges uniformly to $f$ when $\delta$ goes to zero, it also holds true for $f(I)$ for any open interval $I$ containing $x$, which easily implies that $f$ has a critical point flat up to order $r$ at $x$ and thus contradict our assumption.
 \end{proof}
 
Finally it follows from  (\ref{mi95})  that the entropy is continuous  at $f\in D^r([0,1])$ for the $C^1$ topology in any  $C^r$ bounded set (in particular for the $C^r$ topology).  
 
 \end{proof}
\end{appendix}


\begin{thebibliography}{HD}


\normalsize
\baselineskip=17pt

\bibitem{bo}
R. Bowen, {\it Entropy for maps of the interval}, Topology 16 (1977), no. 4, 465-467.

\bibitem{Bow}
R. Bowen, {\it Some systems with unique equilibrium states}, Math. Systems
Theory, 8 (1974/75), 193-202.

\bibitem{bure}
D.Burguet, {\it Examples of $\mathcal{C}^r$ interval map with large symbolic
extension entropy}, Discrete and Continuous dynamical systems A 26 (2010), 872-899.

\bibitem{Bur}
D. Burguet, {\it Existence of measures of maximal entropy for $C^ r$ interval maps}, Proc. Amer. Math. Soc. 142 (2014), no. 3, 957-968.



\bibitem{BLY}
D. Burguet, G.Liao, J. Yang, {\it Asymptotic $h$-expansiveness of $C^\infty$ maps}, Preprint (2013).

\bibitem{bur}
D.Burguet, {\it Examples of $\mathcal{C}^r$ interval map with large symbolic
extension entropy}, Discrete and Continuous dynamical systems A 26 (2010), 872-899.




\bibitem{Buz1}
J. Buzzi,  {\it Intrinsic ergodicity for smooth interval maps},
Israel J. Math., 100 (1997), 125-161.

\bibitem{Buz0}
J.Buzzi, {\it $C^r$ surface diffeomorphisms with no maximal entropy measure}, 
Erg.Th.Dyn.Sytems (to appear).


\bibitem{BuRu}
J.Buzzi, S.Ruette, {\it Large entropy implies existence of a maximal entropy measure for interval maps}, Discrete Contin. Dyn. Syst. Ser. A, 14, No. 4, 673-688, 2006.

\bibitem{Dow}T. Downarowicz, {\it Entropy in dynamical systems}, New Mathematical Monographs, 18. Cambridge University Press, Cambridge, 2011. 

\bibitem{DB}T. Downarowicz and    A. Maass,
\newblock  Smooth interval maps have symbolic extensions,
\newblock  {\em Invent. Math}., {\bf 176}: 617--636, 2009.



\bibitem{Kat}
A. Katok, {\it Lyapounov exponents, entropy and periodic orbits
for diffeomorphisms}, Publ. Math. I.H.E.S. 51 (1980), 137-173.



\bibitem{Koz}O.  Kozlovski, W. Shen, and S.  van Strien,
\newblock Density of hyperbolicity in dimension one,
\newblock {\em Ann. of Math.}, {\bf 166}: 145--182, 2007.

\bibitem{ggg}
B.M. Gurevic, {\it Shift entropy and Markov measures in the path space of a
denumerable graph},  Dokl. Akad. Nauk SSSR, 192: 963-965, 1970.
English translation Soviet. Math. Dokl, 11(3): 744-747, 1970.


\bibitem{ggur}
B.M. Gurevic, {\it Topological entropy of enumerable Markov chains}, 
Dokl. Akad. Nauk SSSR, 187: 715-718, 1969. English translation Soviet.
Math. Dokl, 10(4): 91-915, 1969.

\bibitem{Gur}
B.M. Gurevich, S.V. Savchenko, {\it Thermodynamic formalism for countable symbolic Markov chains}, Russian Math. Surveys, 53 no. 2 (1998), 245-344.

\bibitem{hof} F. Hofbauer, {\it On intrinsic ergodicity of piecewise monotone transformations with positive entropy}, I. Israel J. Math. 34 (1979), no. 3, 213–237; II. Israel J. Math. 38 (1981), no. 1-2, 107–115.



\bibitem{Iw}
K. Iwai,  {\it Continuity of topological entropy of one-dimensional maps with degenerate critical points},
J. Math. Sci. Univ. Tokyo 5 (1998), no. 1, 19-40.


\bibitem{Misho}
M. Misiurewicz, {\it Horseshoes for continuous mappings of an interval}, in "Dynamical Systems", Liguori Editore, Napoli 1980,  127-135.

\bibitem{Mis0}
M. Misiurewicz, {\it On non-continuity of topological entropy}, Bull. Acad. Polon. Sci. Sér. Sci. Math. Astronom. Phys. 19 (1971), 319-320. 

\bibitem{Miss}
M. Misiurewicz, {\it Topological conditional entropy}
Studia Mathematica 55 (1976), 175-200.

\bibitem{Mis1}
M. Misiurewicz, {\it Jumps of entropy in one dimension}, Fund. Math. 132 (1989), no. 3, 215-226. 

\bibitem{Mis2}
M. Misiurewicz and W. Szlenk, { \it Entropy of piecewise monotone mappings}, Studia Math. 67 (1980), no. 1, 45-63. 

\bibitem{Mis3}
M. Misiurewicz, {\it Diffeomorphism without any measure with maximal
entropy}, Bull. Acad. Pol. Sci. 10 (1973), 903-910.

\bibitem{Mis4}
M. Misiurewicz, {\it Continuity of entropy revisited}, Dynamical systems and applications, World
Sci. Ser. Appl. Anal. 4 (1995), 495-503.

\bibitem{New}
S.Newhouse, {\it Continuity properties of entropy}, Annals of
Math. 129 (1989) 215-235.

\bibitem{Rue}
S. Ruette, \emph{Mixing $\mathcal{C}^r$ maps of the interval without maximal measure}, Israel J. Math., \textbf{127} (2002), 253-277.

\bibitem{ver}
D. Vere-Jones, {\it Geometric ergodicity in denumerable Markov chains}, Quart.
J. Math. Oxford Ser. (2), 13-28, 1962.


\bibitem{W}
  P. Walters, {\it An introduction to ergodic theory},
 Graduate Texts in Mathematics, 79. Springer-Verlag, New York-Berlin, 1982.

\bibitem{Yoma}
Y.Yomdin, {\it Volume growth and entropy}, Israel J.Math. 57
(1987), 285-300.



\end{thebibliography}
\end{document}